\documentclass{amsart}

\usepackage{amsmath}
\usepackage{amsthm}
\usepackage{amssymb}
\usepackage{dsfont}
\usepackage{graphicx}
\usepackage{caption}
\usepackage{subcaption}
\usepackage{color}
\usepackage[english]{babel}
\usepackage{layout}

\theoremstyle{plain}
\begingroup
\newtheorem{theorem}{Theorem}[section]
\newtheorem{lemma}[theorem]{Lemma}
\newtheorem{proposition}[theorem]{Proposition}

\endgroup

\theoremstyle{definition}
\begingroup
\newtheorem{definition}[theorem]{Definition}
\newtheorem{remark}[theorem]{Remark}

\endgroup

\theoremstyle{remark}

\mathchardef\emptyset="001F

\numberwithin{equation}{section}

\newcommand{\dt}{\,\mathrm{d}t}

\newcommand{\dist}{{\rm{dist}}}

\def\R{\mathbb{R}}

\def\w{\omega}
\def\a{\alpha}
\def\b{\beta}
\def\g{\gamma}

\begin{document}

\author{Marco Cicalese}
\address[Marco Cicalese]{Zentrum Mathematik - M7, Technische Universit\"at M\"unchen, Boltzmannstrasse 3, 85747 Garching, Germany}
\email{cicalese@ma.tum.de}

\author{Francesco Solombrino}
\address[Francesco Solombrino]{Zentrum Mathematik - M7, Technische Universit\"at M\"unchen, Boltzmannstrasse 3, 85747 Garching, Germany\vspace{-.3cm}}
\address{New: Dip. Mat. Appl. ``Renato Caccioppoli'', Univ. Napoli ``Federico II'', Via Cintia, Monte S. Angelo
80126 Napoli, Italy}
\email{solombri@ma.tum.de}

\author{Matthias Ruf}
\address[Matthias Ruf]{Zentrum Mathematik - M7, Technische Universit\"at M\"unchen, Boltzmannstrasse 3, 85747 Garching, Germany}
\email{mruf@ma.tum.de}

\title{Hemihelical local minimizers in prestrained elastic bi-strips}

\begin{abstract}
We consider a double layered prestrained elastic rod in the limit of vanishing cross section. For the resulting limit Kirchoff-rod model 
with intrinsic curvature we prove a supercritical bifurcation result, rigorously showing the emergence of a branch of hemihelical local minimizers from the straight configuration, at a critical force and under clamping at both ends. As a consequence we obtain the existence of nontrivial local minimizers of the $3$-d system. 
\end{abstract}

\maketitle

\begin{section}{Introduction}

The derivation of the elastic energy of a thin object as limit of the elastic energy of the three-dimensional body when its thickness vanishes has recently gained increasing attention in the case of prestrained bodies. Many authors contributed to this topics, both in the case of $3$-d to $2$-d dimension reduction as in \cite{BMS,  LMP1, LMP2, S} and in the case of $3$-d to $1$-d dimension reduction as in \cite{CRS16, KOB, LA} (see also \cite{ADS, ADSK} for a similar problem in the theory of nematic elastomers). 

As a model example in the $3$-d to $1$-d case in \cite{CRS16}, motivated by recent experiments in \cite{Plos}, we have considered the following model: given two strips of elastomers of the same initial width, but unequal length, one stretches the short one uniaxially to be equal in length to the longer one and then glues them together side-by-side along their length. In such a way a bi-strip system is formed in which the initially shorter strip is under a uniaxial prestrain. The bi-strip is made flat by the presence of a terminal load which is gradually released so that it starts to bend and twist out of plane. According to \cite{Plos} the system may evolve towards either a helical or a hemihelical shape, more complex structure in which helices with different chiralities seem to periodically alternate, leading to the formation of the so-called perversions. 

In the paper \cite{CRS16} we have analyzed the $3$-d to $1$-d limit of the bi-strip system above via $\Gamma$-convergence, rigorously proving that for small prestrains the Kirchoff-rod energy with intrinsic curvature is the right variational approximation of the model. Among other things we have also proved a preliminary bifurcation result, showing the emergence of a branch of nontrivial stationary points from the straight configuration, at a critical force $f_{crit}$ and under clamping at both ends.  Furthermore, along the lines of a result by Kohn and Sternberg (\cite{KS}), we have proved the existence of local minimizer for the $3d$-model arbitrarily close to the limiting ones. Such a result however could not be applied to the branch of critical points we found, since we had no information about their stability.\\ 

In the present paper we fill this gap by a thorough analysis establishing that the branch of stationary points found in \cite{CRS16} is locally made of exactly two {\it non trivial strict local minimizers}, close to the straight configuration, for any given value of the force {\it in a left neighborhood} of the critical force $f_{crit}$ (see the bifurcation diagram in Fig \ref{fig-intro}). Taking into account our sign convention, motivated by the mechanical experiment we have in mind, this corresponds to a supercritical bifurcation diagram. As we discuss in Remark \ref{shape}, these local minimizers have a so-called hemihelical structure. The emergence of such configurations is a well-known experimental fact in the mechanical literature (see \cite[Introduction]{tendril}  for an heuristic description of the phenomenon) which finds now a full mathematical justification. We additionally prove that for those values of the forces, the local minimizers have less energy than the straight configuration, an information that can be important to study evolutionary problems. Furthermore, thanks to our result in \cite{CRS16}, existence of $3$-d nontrivial local minimizers arbitrarily close to these configurations immediately follows (see Theorem \ref{hemihelix}). Coming back to the experiment in \cite{Plos}, it is worth noticing, as we discuss in Remark \ref{multiple}, that our result do not cover the case of multiple perversions for which a different analysis seems to be needed (see also \cite{LA}). \\


\begin{figure}[h]
\centering{
\def\svgwidth{95mm}
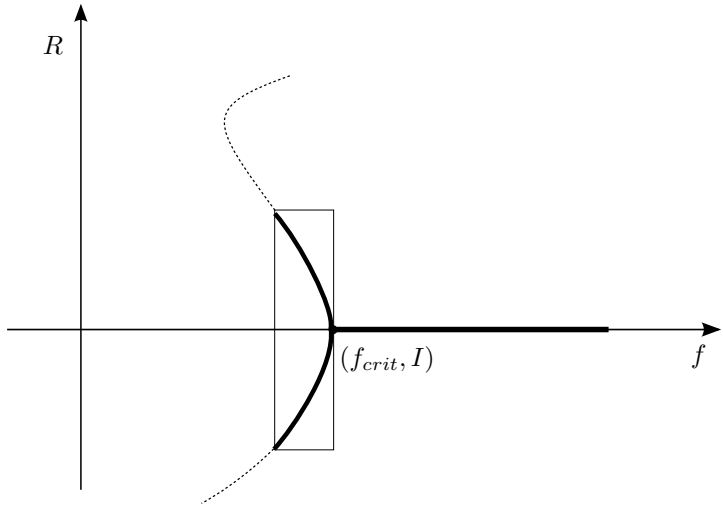
\caption{A local bifurcation diagram of the energy. The bold lines indicates stable configurations.}
	\label{fig-intro}
}
\end{figure}

The main idea behind the proof of our main result (Theorem \ref{localopt}) is a careful combination of variational techniques  with classical results about eigenvalue problems in Banach spaces. In particular we show that if the second differential of the energy along a curve of stationary points is not positive definite, then a nontrivial solution to an eigenvalue problem (see \eqref{stronglagrange}) appears. The behaviour of the smallest eigenvalue of such a problem can be determined using a classical result by Crandall and Rabinowitz (Theorem \ref{CRII}). Eventually this allows us to find a contradiction and then show that our stationary points are indeed strict local minimizers. It is worth pointing out that the abstract approach to bifurcation problems (see \cite{AmPr}) is usually formulated in spaces of smooth functions in order to have high differentiability of the involved functionals, while local minimality of our configurations is meaningful in the natural energy space of Sobolev functions. We fill this gap thanks to the regularity of the solutions of the Euler-Lagrange equation. \\

From a technical point of view, in order to prove these results we need to describe our energy functional in terms of Cardan angles (see Section \ref{section 3}), while for the existence result in \cite[Theorem 5.6]{CRS16} we introduced a different auxiliary functional. This latter, however, would be not useful to perform the more detailed analysis contained in the present paper, since its differentials coincide with those of the energy only at first order, while for some auxiliary computations as those in Lemma \ref{abc} and for the arguments in Theorem \ref{localopt} we need to use differentials of the energy of higher order. Moreover, it is only thanks to the use of Cardan angles that already in Proposition \ref{definebranch} we can obtain more information than in \cite[Theorem 5.6]{CRS16}, namely the local behaviour of the force component of the bifurcation curve, as well as the energy inequality (iii) in the statement. 

\end{section}

\section{Notation and preliminaries}
\subsection{Basic notation}
We denote by $\{e_{1},e_{2},e_{3}\}$ the standard basis in $\R^{3}$, by $\mathbb{M}^{3\times 3}$ the set of all real-valued $3\times3$ matrices and by $I$ the identity matrix. Given a matrix $M$ we denote by $M^{T}$ its transposed matrix (this convention will be also used for row and column vectors). We let $SO(3)\subset\mathbb{M}^{3\times 3}$ be the submanifold of all rotations, while $\mathbb{M}^{3\times 3}_{skew}$ denotes the linear space of the $3\times 3$ skew-symmetric matrices. Given $A\in\mathbb{M}^{3\times 3}_{skew}$ we define $\omega_{A}$ as the unique vector such that $Av=\omega_{A}\times v$ for all $v\in\R^3$. \\
All euclidean spaces will be endowed with the canonical Euclidean norm. 
The symbol $\langle \cdot,\cdot\rangle$ indicates duality products in euclidean or Banach spaces. For an operator $F$ between Banach spaces, we will denote by $N$ its nullspace and by $Rg$ its range. 
The derivative of one-dimensional absolutely continuous functions will be denoted by the prime symbol $^{\prime}$. We will use the standard notation $D$ (resp. $D^2$, $D^3$...) for first (resp. second, third...) order (partial) differentials of operators on Euclidean or Banach spaces, while the symbol $\partial$ will be used for partial derivatives in $\R^n$. Only in this last context, with a slight abuse of notation we will sometimes identify differentials with the corresponding derivatives for notational simplicity.

\subsection{A simple experiment with two-layer prestrained beam and its mathematical modeling}\label{exp}
The situation which we aim to describe in this paper is the following. We deform a two-layer elastic beam of length $L$ and cross section $h S$, where $S$ satisfies some symmetry conditions (see \eqref{sym}) and $h>0$ is a small parameter. We suppose the physical system $(0,L)\times h S$ to be such that the upper layer $(0,L)\times h S^{+}$ is prestrained with a stretching of order $h\chi>0$ in order to match the lower one. We consider the system in its straight configuration $(0,L)\times h S$ to be already at equilibrium under the action of a terminal load $fe_{1}$ applied at $\{L\}\times h S$. Moreover both ends of the beam are kept at a fixed twist by means of a suitable torsional moment (a similar situation is described in \cite{tendril} under the slightly more general assumption that the total twist of the filament is constant). As $h$ tends to $0$, at the onset of instability for the straight configuration, one expects that the mid-fiber of the beam deforms into a new stable configuration showing at least one inversion of curvature. The intrinsic curvature induced by the prestrain would indeed prefer an helical-like configuration which is now forbidden by the boundary conditions. Namely, in order to obtain a helix, one end of the beam must be left free to rotate. \\

We below set the mathematical notation of the problem. Given a small parameter $h>0$, a stripe of thickness $h$, mid-fiber $(0,L)$ and cross section $S\subset\R^2$ is denoted by $\Omega_{h}:=(0,L)\times h S$. On $S$ we will assume that it is a bounded open connected set having unitary area and Lipschitz boundary. We set 
\begin{equation*}
S^{+}:=S\cap\{x_3>0\}.
\end{equation*}
We moreover assume that $S$ satisfies the following symmetry properties:  
\begin{equation}\label{sym}
\begin{split}
\int_Sx_2x_3\,\mathrm{d}x_2\mathrm{d}x_3=&\int_Sx_2\,\mathrm{d}x_2\mathrm{d}x_3=\int_Sx_3\,\mathrm{d}x_2\mathrm{d}x_3=\int_{S^{+}}x_2\,\,\mathrm{d}x_2\mathrm{d}x_3=0.
\end{split}
\end{equation}
 
We consider a hyperelastic material and assume a multiplicative decomposition for the strain (see \cite{LMP1}). Denoting by $W: \mathbb{M}^{3\times 3}\to [0,+\infty]$ the strain energy density, the stored energy of a deformation $u:\Omega_{h}\to \R^3$ is expressed by 
\begin{equation*}
E_{h}(u)=\int_{\Omega_{h}}W(\nabla u(x) \overline {A}_{h}(x))\,\mathrm{d}x.
\end{equation*}
where the prestrain $\overline A_{h}:\Omega_{h}\to\mathbb{M}^{3\times 3}$ is of the form
\begin{equation}\label{twolayer}
\overline A_{h}(x):=\begin{cases}
{\rm{diag}}(1+h\chi,\frac{1}{\sqrt{1+h\chi}},\frac{1}{\sqrt{1+h\chi}}) &\mbox{if $x\in S^+$,}\\
I &\mbox{otherwise,}
\end{cases}
\end{equation}
and $\chi>0$ can be thought of as the effective strength of the stretching.

Throughout the paper we make the following standard assumptions on the density $W$:
\begin{itemize}
	\item[(i)] $W(RF)=W(F)\quad\forall R\in SO(3)\quad$ (frame indifference),
	\item[(ii)] $W(F)\geq c\,\dist^2(F,SO(3))$ and $W(I)=0\quad$ (non-degeneracy),
	\item[(iii)] $W$ is $C^2$ in a neighborhood $U$ of $SO(3)\quad$ (regularity),
	\item[(iv)] $W(FR)=W(F)\quad\forall R\in SO(3)\quad$ (isotropy).
\end{itemize}
In addition to the stored energy we consider an external boundary force that is meant to describe the loading at one end. We set $\Gamma_{h}:=\{L\}\times h S$. Given a force field $\overline{f}_{h}:\Gamma_{h}\to\R^3$ we define the total energy as
\begin{equation*}
\overline{E}_{h}(u)=\int_{\Omega_{h}}W(\nabla u(x)\overline {A}_{h}(x))\,\mathrm{d}x-\int_{\Gamma_{h}}\langle\overline{f}_{h}(x),u(x)\rangle\,\mathrm{d}\mathcal{H}^{2}.
\end{equation*}
As it is customary when dealing with the variational analysis of thin objects, we perform a change of variables to rewrite the energy on a fixed domain. Setting $\Omega=\Omega_1$ and $\Gamma=\Gamma_1$, we define a rescaled deformation field $v:\Omega\to\R^3$, a rescaled prestrain $A_{h}:\Omega\to\mathbb{M}^{3\times 3}$ and a rescaled force $f_{h}:\Gamma\to\R^3$ as
\begin{equation*}
v(x)=u(x_1,h x_2,h x_3),\; A_{h}(x)=\overline {A}_{h}(x_1,h x_2,h x_3),\; f_{h}(x)=\overline{f}_{h}(x_1,h x_2,h x_3).
\end{equation*}
Introducing the rescaled gradient $\nabla_{h}v=\partial_1 v\otimes e_1+\frac{1}{h}(\partial_2 v\otimes e_2+\partial_3v\otimes e_3)$, the energy takes the form $\overline{E}_{h}(u)=h^{2}E_{h}(v)$, where
\begin{equation}\label{defenergy}
E_{h}(v)=
\int_{\Omega}W(\nabla_{h}v(x) A_{h}(x))\,\mathrm{d}x-\int_{\Gamma}\langle f_{h}(x),v(x)\rangle\,\mathrm{d}\mathcal{H}^2.
\end{equation}
We are interested in the case where there exists $f\in\R$ such that
\begin{equation}\label{forcescale}
\frac{f_{h}}{h^2}\rightharpoonup fe_1 \quad\text{in }L^2(\Gamma)
\end{equation}
as $h\to 0$, suggesting a meaningful scaling of $E_{h}$ to be $E_{h}/h^{2}$.
\\
\hspace*{0.5cm}

\noindent On the thin rod we prescribe the following clamped-clamped boundary conditions at both ends, namely 
\begin{equation}\label{clamp}
v(0,x_2,x_3)=
v(L,x_2,x_3)-\int_Sv(L,x_2,x_3)\,\mathrm{d}\mathcal{H}^2=\begin{pmatrix} 0\\ h x_2\\h x_3
\end{pmatrix}.
\end{equation}
To save notation, we introduce the class of admissible deformations as
\begin{equation*}
\mathcal{A}_{h}=\{v\in W^{1,2}(\Omega,\R^3):\;v \text{ satisfies } (\ref{clamp}) \text{ in the sense of traces}\}.
\end{equation*}
In \cite{CRS16} we used $\Gamma$-convergence techniques to derive an effective one-dimensional limit energy when $h\to 0$. It has been proved that, due to (\ref{clamp}), if $R\in W^{1,2}((0,L),SO(3))$ is an $L^{2}$-limit of the rescaled strains $\nabla_{h}v_{h}$, then $R$ has to satisfy the following one-dimensional version of the clamped Dirichlet boundary conditions: 
\begin{equation}\label{clamped}
R(0)=R(L)=I;
\end{equation}
as a consequence we introduce the family of those limiting configurations with finite energy to be defined as  
\begin{equation*}
\mathcal{A}:=\{R\in W^{1,2}((0,L),SO(3)):\;R \text{ satisfies }(\ref{clamped})\text{ in the sense of traces}\}.
\end{equation*}
Note that with this notation the limit deformation $v\in W^{2,2}((0,L),\R^3)$ can be recovered via the formula $v(x_1)=\int_0^{x_1}R(t)e_1\,\mathrm{d}t$. In the setting introduced above, up to an additive constant the limit energy takes the form
\begin{equation}\label{1d-energy}
E^{f}_0(R)=\frac{1}{2}\int_0^Lc_{12}a_{12}(t)^2+c_{13}(a_{13}(t)-k)^2+c_{23}a_{23}(t)^2-2f\langle e_1,R(t)e_1\rangle\,\mathrm{d}t,
\end{equation}
where $A(t)=R^T(t)R^{\prime}(t)$. The constants $c_{12},c_{13},c_{23}$ depend on the coefficients of the quadratic form of linearized elasticity for the energy density $W$ and on the geometry of $S$, while $k$ encodes the intrinsic curvature caused by the two-layer structure of the prestrain (see Proposition 3.8 in \cite{CRS16}). 

We now recall one of the main results of \cite{CRS16} that connects isolated local minimizers of the reduced energy functional to local minimizers of the full three-dimensional model. 
In the statement below local minimality of a deformation $v$ has to be understood in the following sense: we say that $v$ is a local minimizer of $E_{h}$ if there exists $\delta>0$ such that $E_{h}(v)\leq E_{h}(w)$ for all $w$ such that $\|\nabla_{h}v-\nabla_{h} w\|_{L^{2}}\leq \delta$. Note that, as shown in \cite[Proposition 3.3]{CRS16}, such a definition arises naturally on the sublevel sets of $\frac1{h^2} E_{h}$.

\begin{theorem}\label{conv-loc}
Assume that the functional $E_{h}$ defined in (\ref{defenergy}) is lower semicontinuous with respect to weak convergence in $W^{1,2}(\Omega,\R^3)$. Moreover let $E^f_0$ be defined as in \eqref{1d-energy} and $R\in W^{1,2}((0,L),SO(3))$ be a strict local minimizer of $E^f_0$ in ${\mathcal A}$ with respect to the strong $W^{1,2}$-topology. Then there exists a sequence $v_{h}$ of local minimizers of $E_{h}$ in $\mathcal A_{h}$ such that $v_{h}\to v$ strongly in $W^{1,2}(\Omega,\R^3)$ and $\nabla_{h} v_{h}\to R$ strongly in $L^2(\Omega,\mathbb{M}^{3\times 3})$.
\end{theorem}

In \cite{CRS16} we also analyzed the local minimality of the straight configuration for the limit model \eqref{1d-energy}. Here we recall also this result under the additional assumption that the following inequality holds:
\begin{align}\label{ass_0}
\frac{(c_{13}k)^2}{c_{23}}-\frac{4\pi^2c_{12}}{L^2}>0.
\end{align}
Such an inequality, which we assume to be true throughout the paper, implies that the critical force, at which local minimality of the straight configuration gets lost, is positive (a natural condition in the experiment described at the beginning of the section).
\begin{theorem}[\cite{CRS16}, Theorem 5.2]\label{idstable?}
Let $E_{0}^{f}$ be as in \eqref{1d-energy}, set $f_{crit}=\frac{(c_{13}k)^2}{c_{23}}-\frac{4\pi^2c_{12}}{L^2}$ and assume \eqref{ass_0}. Then for $f>f_{crit}$ the straight configuration $R(t)\equiv I$ is a strict local minimizer of $E^{f}_{0}$ in the $L^{2}$-topology with the boundary conditions \eqref{clamped}.
If instead $f<f_{crit}$ the straight configuration is not a local minimizer.
\end{theorem}

\subsection{Abstract bifurcation results}
We now recall two abstract results concerning existence and stability of bifurcation branches, that are fundamental to our analysis. 
The first one is an existence result proved in \cite{CrRa} which we state below:
\begin{theorem}[Crandall-Rabinowitz, 1971]\label{CR}
Let $X,Y$ be Banach spaces, $V$ a neighbourhood of $(0,\lambda_0)$ in $X\times\R$ and $F:V\to Y$ have the properties
	
\begin{itemize}
	\item[(i)] $F(0,\lambda)=0$,
	\item[(ii)] the partial derivatives $D_{\varphi}F,D_{\lambda}F,D^2_{\varphi,\lambda}F$ exist and are continuous,
	\item[(iii)] $D_{\varphi}F(0,\lambda_0)$ is a Fredholm operator with zero index and $N(D_{\varphi}F(0,\lambda_0))=\text{span}\{v\}$,
	\item[(iv)] $D^2_{\varphi,\lambda}F(0,\lambda_0)v\notin Rg(D_{\varphi}F(0,\lambda_0))$.
\end{itemize}
If $Z$ is any complement of $N(D_{\varphi}F(0,\lambda_0))$ in $X$, then there is a neighbourhood $U$ of $(0,\lambda_0)$ in $X\times \R$, an interval $(-a,a)$ and continuous functions $\lambda:(-\delta,\delta)\to\R$ and $\psi:(-\delta,\delta)\to Z$ such that $\lambda(0)=0,\,\psi(0)=0$ and
\begin{equation}\label{intorno}
F^{-1}(0)\cap U=\{(sv+s\psi(s),\lambda_0+\lambda(s)):\;|s|<\delta\}\cup \{(0,s):\;(0,s)\in U\}.
\end{equation}
If $F\in C^n(V)$, then $\psi, \lambda,\in C^{n-1}((-\delta,\delta))$.
\end{theorem}

In order to state the second one, proved in \cite[Corollary 1.13 and Theorem 1.16]{CrRa73}, we need to recall the following Definition.

\begin{definition}\label{simple}
Let $X,Y$ be Banach spaces and let $T,K\in L(X,Y)$ be bounded, linear operators. Then $\mu\in\R$ is called a $K$-simple eigenvalue of $T$ if ${\rm dim}(N(T-\mu K))={\rm codim}(Rg(T-\mu K))=1$ and, if $N(T-\mu K)={\rm span}\{v_0\}$, then $Kv_0\notin Rg(T-\mu K)$.
\end{definition}

\begin{theorem}[Crandall-Rabinowitz, 1973]\label{CRII}
In the setting of Theorem \ref{CR}, let $K\in L(X,Y)$ and assume that $0$ is a $K$-simple eigenvalue of $D_{\varphi}F(0,\lambda_0)$. Then there exist open intervals $A,B\subset\R$ such that $\lambda_0\in A$ and $0\in B$ and continuously differentiable functions $\gamma:A\to\R$, $\mu:B\to\R$, $u:A\to X$ and $w:B\to X$ such that
\begin{align}
&D_{\varphi}F(0,\lambda)u(\lambda)=\gamma(\lambda)Ku(\lambda)&\text{ for all }\lambda\in A\nonumber,
\\
&D_{\varphi}F(sv_0+s\psi(s),\lambda_0+\lambda(s))w(s)=\mu(s)Kw(s)&\text{ for all }s\in B.\label{eigenprob}
\end{align}
It holds that $\gamma(\lambda_0)=\mu(0)=0$, $u(\lambda_0)=w(0)=v_0$ and $u(\lambda)-v_0\in Z$ as well as $w(s)-v_0\in Z$.
\\
Moreover we have $\gamma^{\prime}(\lambda_0)\neq 0$ and for $|s|$ small enough the functions $\mu(s)$ and $-s\lambda^{\prime}(s)\gamma^{\prime}(\lambda_0)$ have the same zeros and the same sign in the sense that
\begin{equation*}
\lim_{\substack{s\to 0\\ \mu(s)\neq 0}}\frac{-s\lambda^{\prime}(s)\gamma^{\prime}(\lambda_0)}{\mu(s)}=1.
\end{equation*}	
\end{theorem}

The curve $(\mu(s), w(s))$ satisfying the eigenvalue problem \eqref{eigenprob} is locally uniquely determined by the operator, in a sense made precise by the following Lemma, also proved in \cite{CrRa73}.
\begin{lemma}[\cite{CrRa73}, Lemma 1.3]\label{quasiuniqueness}
Let $T_0,K\in L(X,Y)$ be bounded, linear operators and $\mu_0$ be a $K$-simple eigenvalue of $T_0$, with $N(T_0-\mu_0 K)={\rm span}\{v_0\}$. Then there exist a neigborhood $U_1$ of $T_0$ in $ L(X,Y)$ and a neighborhood $U_2$ of $\mu_0$ in $\R$ such that, for all $T\in U_1$, $T-\mu K$ is singular for a unique $\mu \in U_2$. Furthermore, $\mu$ depends smoothly on $T$ and is itself a $K$-simple eigenvalue.\\
If $Z$ is a complement of  ${\rm span}\{v_0\}$ in $X$ , given $(T, \mu) \in U_1\times U_2$ such that $T-\mu K$ is singular, there exists a unique vector $w$ such that $w-v_0 \in Z$ and
\[
T w= \mu Kw\,.
\]
Also $w$ is a smooth function of $T$.
\end{lemma}

\subsection{Regularity of stationary points}
In our analysis, the functional $F$ appearing in the previous abstract results will be, roughly speaking, given by the first variation of the energy \eqref{1d-energy} (up to rewriting it in suitable local coordinates) and the sign of $\lambda'(s)$ will be determined by using the implicit function theorem along the lines of \cite[Section 5.4]{AmPr}. This will require some additional differentiability of the functional, which can be ensured in stronger topologies than the one of $W^{1,2}$. To this end, we need to show a priori that local minimizers of the energy are indeed regular.

To see this, we first recall the formula for the first variation of the energy as proved in \cite{CRS16}.
Using the short-hand $\mathbf{C}=\text{diag}(c_{23},c_{13},c_{12})$, stationary points of \eqref{1d-energy} satisfy the  integral equality
\begin{align}\label{weakform}
0&=\int_0^L\langle\mathbf{C}\omega_A(t)-c_{13}ke_2,R^T(t)\omega_{B^{\prime}}(t)\rangle-f\langle e_1,(\omega_{B}(t)\times R(t)e_1)\rangle\,\mathrm{d}t\nonumber 
\\
&=\int_0^L\langle\mathbf{C}\omega_A(t)-c_{13}ke_2,R^T(t)\omega_{B^{\prime}}(t)\rangle+f\langle R^T(t)e_1\times e_1,R^T(t)\w_B(t)\rangle\,\mathrm{d}t
\end{align}
for all $B\in W^{1,2}_0((0,L),\mathbb{M}^{3\times 3}_{skew})$
This gives $C^\infty$-regularity of stationary points, which we state and prove in the next lemma for the sake of completeness.
\begin{lemma}\label{regularity}
Let $R\in W^{1,2}((0,L),SO(3))$ satisfy (\ref{clamped}) and (\ref{weakform}). Then $R\in W^{k,1}((0,L),SO(3))$ for every $k\in\mathbb{N}$. In particular $R\in C^{\infty}([0,L],SO(3))$.
\end{lemma}
\begin{proof}
With the admissible ansatz $\w_B=R\varphi$ with $\varphi\in W^{1,2}_0((0,L),\R^3)$, using that $R^{\prime}=RA$, the integral equality (\ref{weakform}) becomes
\begin{equation*}
0=\int_0^L\langle\mathbf{C}\omega_A(t),\varphi^{\prime}(t)\rangle+\langle c_{13}kA(t)e_2-A(t)\mathbf{C}\w_A(t)+fR^T(t)e_1\times e_1,\varphi(t)\rangle\,\mathrm{d}t.
\end{equation*}
Note that by H\"older's inequality the left entry in the second scalar product belongs to $L^1((0,L))$, so that by definition of distributional derivatives it follows that $\w_A\in W^{1,1}((0,L),\R^3)$ and, again by $R^{\prime}=RA$, we deduce $R\in W^{2,1}((0,L),SO(3))$. Inductively we conclude that $R\in W^{k,1}((0,L),SO(3))$ for every $k$ and by the Sobolev embedding we conclude that $R\in C^{\infty}([0,L],SO(3))$. 
\end{proof}
By the previous lemma stationary points and local minimizers must be regular and solve the system of ODEs
\begin{equation}\label{E-L-ODE}
\w_A^{\prime}(t)=c_{13}k\mathbf{C}^{-1}A(t)e_2-\mathbf{C}^{-1}A(t)\mathbf{C}\w_A(t)+f\mathbf{C}^{-1}(R^T(t)e_1\times e_1).
\end{equation}

\section{The energy in local coordinates}\label{section 3}
In order to work in a linear space, instead than on the manifold $W^{1,2}((0,L), SO(3))$, we will preliminarily rewrite the energy in local coordinates in an $L^\infty$ neighborhood of $I$. To this end we recall the notion of Cardan angles that we use as parameters. We namely define $\mathcal{G}:\R^3\to SO(3)$ as
\begin{align*}
\mathcal{G}&=\mathcal{G}(\alpha,\beta,\gamma)
\\
&=\begin{pmatrix}
\cos(\beta)\cos(\gamma) &-\cos(\b)\sin(\gamma) &\sin(\beta) \\
\sin(\a)\sin(\b)\cos(\g)+\cos(\a)\sin(\g) &\cos(\a)\cos(\g)-\sin(\a)\sin(\b)\sin(\g) &-\sin(\a)\cos(\b)\\
\sin(\a)\sin(\g)-\cos(\a)\sin(\b)\cos(\g) &\sin(\a)\cos(\g)+\cos(\a)\sin(\b)\sin(\g) &\cos(\a)\cos(\b)
\end{pmatrix}.
\end{align*}
It is well-known that $\mathcal{G}$, when restricted to $U=(-\pi,\pi)\times (-\pi/2,\pi/2)\times(-\pi,\pi)$, is a diffeomorphism from $U$ onto an open neighbourhood of $I$ in $SO(3)$.
With a slight abuse of notation we will also denote with $\mathcal G$ the induced mapping from $W^{1,2}((0,L),\R^3)$ to $W^{1,2}((0,L),SO(3))$. Its properties are summarized in the following lemma.
\begin{lemma}\label{inverse}
There exists $\delta>0$ such that for each $R\in W^{1,2}((0,L),SO(3))$ with $\|R-I\|_{\infty}<\delta$ there exists $\varphi\in W^{1,2}((0,L),\R^3)$ with $R=\mathcal{G}(\varphi)$. The function $\varphi=\mathcal{G}^{-1}(R)$ inherits the differentiability properties of $R$. If $R$ additionally satisfies (\ref{clamped}), then $\varphi\in W^{1,2}_0((0,L),\R^3)$.
\end{lemma}
\begin{proof}
Since $\mathcal{G}$ is invertible with smooth inverse in a neighbourhood of $I$, the claim follows by the chain rule for Sobolev functions.
\end{proof}

With the previous lemma at hand, we rewrite the energy in terms of Cardan angles. A direct computation yields that, for a vector-valued function $\varphi=(\varphi_1,\varphi_2,\varphi_3)$, the components of the skew-symmetric matrix $\mathcal{G}^T(\varphi(t))\mathcal{G}(\varphi(t))^{\prime}$ are given by
\begin{align*}
&a_{12}(t)=-\varphi_1^{\prime}(t)\sin(\varphi_2(t))-\varphi_3^{\prime}(t),\\
&a_{13}(t)=-\varphi_1^{\prime}(t)\cos(\varphi_2(t))\sin(\varphi_3(t))+\varphi_2^{\prime}(t)\cos(\varphi_3(t)),\\
&a_{23}(t)=-\varphi_1^{\prime}(t)\cos(\varphi_2(t))\cos(\varphi_3(t))-\varphi_2^{\prime}(t)\sin(\varphi_3(t)).
\end{align*}
Hence, the one-dimensional energy in (\ref{1d-energy}) can be written in terms of Cardan angles as
\begin{align}\label{angleenergy}
E_0^f(\varphi)=&\frac{1}{2}\int_0^Lc_{12}\left(\varphi_1^{\prime}(t)\sin(\varphi_2(t))+\varphi_3^{\prime}(t)\right)^2+c_{13}\left(\varphi_1^{\prime}(t)\cos(\varphi_2(t))\sin(\varphi_3(t))-\varphi_2^{\prime}(t)\cos(\varphi_3(t))+k\right)^2\dt\nonumber \\
&+\frac{1}{2}\int_0^Lc_{23}\left(\varphi_1^{\prime}(t)\cos(\varphi_2(t))\cos(\varphi_3(t))+\varphi_2^{\prime}(t)\sin(\varphi_3(t))\right)^2-2f\cos(\varphi_2(t))\cos(\varphi_3(t))\,\dt.
\end{align}
For notational convenience we introduce the integrand $g_f:\R^3\times\R^3\to\R$ setting
\begin{align*}
g_f(u,\xi)=&\frac{c_{12}}{2}(\xi_1\sin(u_2)+\xi_3)^2+\frac{c_{13}}{2}(\xi_1\cos(u_2)\sin(u_3)-\xi_2\cos(u_3)+k)^2\\
&+\frac{c_{23}}{2}(\xi_1\cos(u_2)\cos(u_3)+\xi_2\sin(u_3))^2-f\cos(u_2)\cos(u_3).
\end{align*}
Notice that this integrand is quadratic in $\xi$ and satisfies all the assumptions in \cite[Lemma 4.5]{CRS16}. Therefore, the same proof yields the following differentiability property:
\begin{align}\label{frechet}
E_0^f\in C^2(W^{1,2}((0,L),\R^3), \R).
\end{align}

In order to study the behaviour of the energy close to the critical force, we start with a bifurcation analysis of the angular-energy (\ref{angleenergy}). To this end, we need the associated Euler-Lagrange equation given by 
\begin{equation}\label{EL-angle}
\left({\nabla}_{\xi}g_f(\varphi(t),\varphi^{\prime}(t))\right)^{\prime}=\nabla_ug_f(\varphi(t),\varphi^{\prime}(t))
\end{equation}
\begin{remark}\label{equivequations}
Let us observe that if $\varphi\in C_0^2([0,L],\R^3)$ is a strong solution of the system above, then the function $\mathcal{G}(\varphi)$ is a stationary point of the functional $E_0^f$ in (\ref{1d-energy}), that is it satisfies (\ref{weakform}). Indeed, any curve of admissible deformations that is tangential to $\mathcal{G}(\varphi)$ can be transformed via Lemma \ref{inverse} to a tangential curve of Cardan angles. Conversely, by the Lemmata \ref{regularity} and \ref{inverse}, any stationary point of the functional (\ref{1d-energy}) which is sufficiently close to the identity yields Cardan angles that are a regular solutions of (\ref{EL-angle}).
\end{remark}
We now set $f=\lambda$ as bifurcation parameter, and study the operator $F:C_0^2([0,L],\R^3)\times\R\to C([0,L],\R^3)$ defined as
\begin{equation}\label{defoperator}
F(\varphi,\lambda)=\left({\nabla}_{\xi}g_{\lambda}(\varphi(t),\varphi^{\prime}(t))\right)^{\prime}-\nabla_ug_{\lambda}(\varphi(t),\varphi^{\prime}(t)).
\end{equation}
Note that, by definition, for any $\varphi\in C^2_0([0,L],\R^3)$ and every $w\in W^{1,2}_0((0,L),\R^3)$ it holds
\begin{equation}\label{minussign}
\langle F(\varphi,\lambda),w\rangle=-\langle DE_0^{\lambda}(\varphi),w\rangle.
\end{equation}
On its domain this functional is very regular in the sense of Fr{\'e}chet-differentiability as stated in the lemma below. 
\begin{lemma}\label{diff}
The operator $F:C_0^2([0,L],\R^3)\times\R\to C([0,L],\R^3)$ is  $C^\infty$.
\end{lemma}
\begin{proof} The proof is left to the reader.
\end{proof}
We set
\begin{align}\label{lambda_0}
\lambda_{0}=\frac{(c_{13}k)^2}{c_{23}}-\frac{4\pi^2c_{12}}{L^2}
\end{align}
and 
\begin{equation}\label{kernel}
w^{*}(t)=\begin{pmatrix}
\frac{c_{13}kL}{2c_{23}\pi} \left(1-\cos(\frac{2\pi}{L}t)\right)
\\
0
\\
-\sin(\frac{2\pi}{L}t)
\end{pmatrix}.
\end{equation}

We remark that, because of \eqref{ass_0}, $\lambda_{0}>0$. We also note that the third component of $w^{*}$ has a change of sign in $[0,L]$, which is due to considering clamped boundary conditions in the eigenvalue problem for $D_{\varphi}F$ and will eventually lead to an inversion of curvature of the bifurcating stable configurations (see Fig \ref{fig1}).\\

As a first step of our analysis, we show that the operator $F$ fulfills the assumption of the bifurcation theorem for $\lambda=\lambda_0$.
\begin{lemma}\label{bifurcation}
Let $\lambda_0$ and $w^{*}\in C_0^2([0,L],\R^3)$ be given by \eqref{lambda_0}, and \eqref{kernel}, respectively. Then the operator $F$ defined in \eqref{defoperator} satisfies the assumptions of Theorem \ref{CR} with $v=w^{*}$.
\end{lemma}
\begin{proof}
By a direct computation $\nabla_{\xi}g_{\lambda}(0,0)=\nabla_ug_{\lambda}(0,0)=0$ for all $\lambda$. Hence $F(0,\lambda)=0$. Moreover, by Lemma \ref{diff} the operator fulfills the differentiability assumptions (ii). Next, let us calculate the first derivative. We have that
\begin{equation}\label{derivativeoperator}
D_{\varphi}F(\varphi,\lambda)w=\left(\partial_{\xi}\partial_{u}g_{\lambda}(\varphi,\varphi^{\prime})w+\partial^2_{\xi}g_{\lambda}(\varphi,\varphi^{\prime})w^{\prime}\right)^{\prime}-\partial^2_{u}g_{\lambda}(\varphi,\varphi^{\prime})w-\partial_{u}\partial_{\xi}g_{\lambda}(\varphi,\varphi^{\prime})w^{\prime}.
\end{equation}
Plugging in $\varphi=0$ we obtain by a straightforward calculation that
\begin{equation*}
D_{\varphi}F(0,\lambda)w=c_{12}w_3^{\prime\prime}e_3+c_{13}w_2^{\prime\prime}e_2+c_{13}kw_3^{\prime}e_1-c_{13}kw_1^{\prime}e_3+c_{23}w_1^{\prime\prime}e_1-\lambda w_2e_2-\lambda w_3e_3.
\end{equation*}
Note that the bounded, linear operator $T:C_0^2([0,L],\R^3)\to C([0,L],\R^3)$ defined by $Tw=\mathbf{C}w^{\prime\prime}$ is bijective. Hence $D_{\varphi}F(0,\lambda_0)$ is a compact perturbation of a bijective operator, whence a Fredholm-operator of index zero. A direct computation shows that $w^{*}$ given by (\ref{kernel}) satisfies $D_{\varphi}F(0,\lambda_0)w^{*}=0$. In order to determine the dimension of the kernel, we note that the equation $D_{\varphi}F(0,\lambda_0)w=0$ is a system of second order linear differential equations. The second component $w_2$ must satisfy
\[
c_{13}w_2^{\prime\prime}-\lambda_0 w_2=0.
\]
Since we assume $\lambda_0>0$, by the Dirichlet boundary conditions we immediately get $w_2\equiv0$. For the remaining components one can write the equation as a first order four-dimensional system with the constant matrix
\begin{equation*}
A=\begin{pmatrix}
0 &-c_{13}k/c_{23} &0 &0\\ c_{13}k/c_{12} &0 &0 &\lambda_0/c_{12}\\ 1&0&0&0 \\ 0&1&0&0
\end{pmatrix}.
\end{equation*}
This matrix as the eigenvalues $\left\{0,\pm\sqrt{\frac{\lambda_0}{c_{12}}-\frac{(c_{13}k)^2}{c_{12}c_{23}}}\right\}$, where $0$ has algebraic multiplicity $2$. Hence the solutions are of the form
\begin{equation*}
\begin{split}
&w_1(t)=a_1+a_2t+a_3\exp(2\pi it/L)+a_4\exp(-2\pi it/L),\\
&w_3(t)=b_1+b_2t+b_3\exp(2\pi it/L)+b_4\exp(-2\pi it/L),
\end{split}
\end{equation*}
with $a_i,b_i\in\mathbb{C}$. Plugging this ansatz into the equation and comparing the coefficients of the independent functions, together with the boundary conditions we obtain a $8$-dimensional linear system that can be solved explicitly for a one-dimensional kernel spanned by the function $w^{*}$ in \eqref{kernel}. 

To show the transversality condition $D^2_{\varphi,\lambda}F(0,\lambda_{0})w^{*}\notin Rg(D_{\varphi}F(0,\lambda_{0}))$, we argue by contradiction. Then there exists a solution of the system
\begin{equation*}
c_{12}w_3^{\prime\prime}e_3+c_{13}w_2^{\prime\prime}e_2+c_{13}kw_3^{\prime}e_1-c_{13}kw_1^{\prime}e_3+c_{23}w_1^{\prime\prime}e_1-\lambda_0 w_2e_2-\lambda_0 w_3e_3=\lambda_0\sin\left(\frac{2\pi}{L}t\right)e_3.
\end{equation*}
Integrating the first component we infer that $c_{23}w_1^{\prime}=-c_{13}kw_3-C$ for some constant $C\in\R$. With this formula we can rewrite the third component via
\begin{equation*}
c_{12}w_3^{\prime\prime}+\frac{c_{13}k}{c_{23}}(c_{13}kw_3+C)-\lambda_0w_3=\lambda_0\sin\left(\frac{2\pi}{L}t\right).
\end{equation*}
Multiplying the equation with the right hand side and integrating twice by parts over $(0,L)$ we obtain
\begin{equation*}
0<\lambda_0\int_0^L-\frac{4\pi^2c_{12}}{L^2}w_3\sin\left(\frac{2\pi}{L}t\right)+\left(\frac{(c_{13}k)^2}{c_{23}}-\lambda_0\right)w_3\sin\left(\frac{2\pi}{L}t\right)\,\dt=0,
\end{equation*}
where we used the definition of $\lambda_0$ and the boundary conditions on $w_3$. This gives the desired contradiction.
\end{proof}
\begin{remark}\label{multiple}
With a similar analysis as the one in the previous lemma, one can also find other eigenvalues of the operator $D_{\varphi}F$, corresponding to smaller forces than the critical one. The corresponding eigenstates show more than one sign change in the third component, that would lead to multiple inversions of curvature (see \cite{Plos} for experimental evidence). On the other hand, possible bifurcation curves starting from the straight configuration along those directions are, at least close to the identity, not made of local minimizers. This can be proved exploiting that, according to Theorem \ref{idstable?}, the second differential of $E_0^f$ at the identity is not positive semidefinite for $f<f_{crit}$, together with a lower semicontinuity argument. This is not the case for the bifurcation branch from the largest eigenvalue $\lambda_0=f_{crit}$. Indeed we will show in Theorem \ref{localopt} that such a branch in a neighborhood of the straight configuration consists of local minimizers.
\end{remark}
The previous lemma and the Crandall-Rabinowitz Theorem \ref{CR} entail the existence of a branch of solutions bifurcating the identity at the critical force. A precise statement with additional properties will be given in the next section. 
In order to investigate further the behaviour of the non-trivial branch, we will follow the general approach described in \cite[Section 5.4]{AmPr}. To this end, we first notice that 
\begin{equation}\label{complement}
Rg(D_{\varphi}F(0,\lambda_0))=\{w\in C([0,L],\R^3):\;\langle w^{*},w\rangle=0\}\,,
\end{equation}
where we identify the function $w^{*}$ in \eqref{hemihelix} with an absolutely continuous vector-valued measure as usual. The range has indeed codimension $1$ by the previous lemma, while a direct computation based on integration by parts gives for all $w\in C_0^2([0,L],\R^3)$ 
\begin{equation}\label{orthogonal}
\langle w^{*},D_{\varphi}F(0,\lambda_{0})w\rangle=\int_0^Lw^{*}(t)^TD_{\varphi}F(0,\lambda_{0})w(t)\,\mathrm{d}t=\int_0^Lw(t)^TD_{\varphi}F(0,\lambda_{0})w^{*}(t)\,\mathrm{d}t=0.
\end{equation}
In order to determine the type of bifurcation, in the next lemma we compute the following terms:
\begin{align}
&a:=\langle w^{*},D^2_{\varphi,\lambda}F(0,\lambda_{0})w^{*}\rangle,\label{a}
\\
&b:=\frac{1}{2}\langle w^{*},D^2_{\varphi,\varphi}F(0,\lambda_{0})[w^{*},w^{*}]\rangle,\label{b}
\\
&c:=-\frac{1}{3a}\langle w^{*},D^3_{\varphi,\varphi,\varphi}F(0,\lambda_{0})[w^{*},w^{*},w^{*}]\rangle.\label{c}
\end{align}

\begin{lemma}\label{abc}
For $F$  as in \eqref{defoperator} and $\lambda_0$ as in \eqref{lambda_0}, it holds
\begin{align*}
&a=-\frac{L}{2},\\
&b=0,\\
&c=-\left(\frac{3(c_{13}-c_{23})(c_{13}k)^2}{c^2_{23}}+\frac{9(c_{13}k)^2}{2c_{23}}-\frac{\lambda_0}{4}\right)=-\left(\frac{(3c_{13}+\frac{5}{4}c_{23})(c_{13}k)^2}{c^2_{23}}+\frac{\pi^2c_{12}}{L^2}\right).
\end{align*}
In particular $c<0$.
\end{lemma}
\begin{proof}
First note that $D^2_{\varphi,\lambda}(0,\lambda_0)w^{*}=-w^{*}_3e_3$. Then it holds that
\begin{equation*}
a=-\int_0^L|w^{*}_3(t)|^2\,\dt=-\int_0^L\sin^2\left(\frac{2\pi}{L}t\right)\,\dt=-\frac{L}{2}.
\end{equation*}
In order to calculate $b$ and $c$, we first write the operator $F$ in components. Note that we need only the first and the third component since $w^{*}_2=0$. By linearity of differentiation, it holds that
\begin{equation*}
\langle D^2_{\varphi,\varphi}F(0,\lambda)[w^*,w^*],e_1\rangle=\Big(D^2\partial_{\xi_1}g_{\lambda}(0,0)[(w^*,(w^*)^{\prime}),(w^*,(w^*)^{\prime})]\Big)^{\prime},
\end{equation*}
where the symbol $D^2$ on the right hand side denotes the Hessian of the scalar function $\partial_{\xi_1}g_{\lambda}$. Observe that $\partial_{\xi_1}g_{\lambda}(u,\xi)$ reads as
\begin{align*}
\partial_{\xi_1}g_{\lambda}(u,\xi)=&c_{12}\Big(\xi_1\sin(u_2)+\xi_3\Big)\sin(u_2)+c_{13}\Big(\xi_1\cos(u_2)\sin(u_3)-\xi_2\cos(u_3)+k\Big)\cos(u_2)\sin(u_3)
\\ &+c_{23}\Big(\xi_1\cos(u_2)\cos(u_3)+\xi_2\sin(u_3)\Big)\cos(u_2)\cos(u_3).
\end{align*}
Note that due to the fact that $w^{*}_2=0$, any higher order derivative of $\partial_{\xi_1}g_{\lambda}(u, \xi)$ with at least one derivative with respect to the variables $u_2$ or $\xi_2$ along the direction $w^{*}$ vanishes. Since $\partial_{\xi_3}\partial_{\xi_1}g_{\lambda}(u,\xi)=c_{12}\sin(u_2)$, the same reasoning allows to neglect any higher order derivatives with at least one derivative with respect to $\xi_3$. Finally, the function $\partial_{\xi_1}g_{\lambda}(u,\xi)$ is independent of the variable $u_1$. Summarizing we need to take into account only the derivatives of $\partial_{\xi_1}g_{\lambda}(u, \xi)$ with respect to $u_3$ and $\xi_1$. A straightforward calculation shows that
\begin{equation*}
\partial^2_{u_3}\partial_{\xi_1}g_{\lambda}(0,0)=\partial_{u_3}\partial_{\xi_1}^2g_{\lambda}(0,0)=\partial_{\xi_1}^3g_{\lambda}(0,0)=0,
\end{equation*}
so that it follows directly that
\begin{equation}\label{bfirstcomponent}
\langle D^2_{\varphi,\varphi}F(0,\lambda)[w^{*},w^{*}],e_1\rangle=0
\end{equation}
for all $\lambda$. For the third component we need to compute 
\begin{equation*}
\langle D^2_{\varphi,\varphi}F(0,\lambda)[w^*,w^*],e_3\rangle=\Big(D^2\partial_{\xi_3}g_{\lambda}(0,0)[(w^*,(w^*)^{\prime}),(w^*,(w^*)^{\prime})]\Big)^{\prime}-D^2\partial_{u_3}g_{\lambda}(0,0)[(w^*,(w^*)^{\prime}),(w^*,(w^*)^{\prime})].
\end{equation*}
Using the explicit expressions
\begin{align*}
\partial_{\xi_3}g_{\lambda}(u,\xi)=&c_{12}(\xi_1\sin(u_2)+\xi_3), 
 \\
 -\partial_{u_3}g_{\lambda}(u,\xi)=&-f\cos(u_2)\sin(u_3)
\\
&-c_{13}\left(\xi_1\cos(u_2)\sin(u_3)-\xi_2\cos(u_3)+k\right)\left(\xi_1\cos(u_2)\cos(u_3)+\xi_2\sin(u_3)\right)
\\
&-c_{23}\left(\xi_1\cos(u_2)\cos(u_3)+\xi_2\sin(u_3)\right)\left(-\xi_1\cos(u_2)\sin(u_3)+\xi_2\cos(u_3)\right)
\end{align*}
and arguing as for the first component, for second or higher order derivatives along the direction $w^{*}$ it suffices to consider partial derivatives with respect to the variables $u_3$ and $\xi_1$. Again we obtain
\begin{equation}\label{bthirdcomponent}
\langle D^2_{\varphi,\varphi}F(0,\lambda)[w^{*},w^{*}],e_3\rangle=0.
\end{equation}
for all $\lambda$. Combining \eqref{bfirstcomponent} and \eqref{bthirdcomponent} we obtain that $b=0$. 

In order to compute $c$, for the first component we need to take third order derivatives of $\partial_{\xi_1}g_{\lambda}(u, \xi)$ again only in the variables $u_3$ and $\xi_1$. Since $\partial_{\xi_1}g_{\lambda}$ is linear in $\xi_1$, only $\partial^3_{u_3}\partial_{\xi_1}g_{\lambda}$ and $\partial_{u_3}^2\partial_{\xi_1}^2g_{\lambda}$ are non-zero. One finds that
\begin{equation}\label{cfirstcomponent}
\begin{split}
&\partial_{u_3}^2\partial_{\xi_1}^2g_{\lambda}(0,0)=2(c_{13}-c_{23}),
\\
&\partial^3_{u_3}\partial_{\xi_1}g_{\lambda}(0,0)=-c_{13}k.
\end{split}
\end{equation}
Concerning the third component, again by the linear dependence on the components of the variable $\xi$, non-trivial contributions come only from the terms $\partial_{u_3}^4g_{\lambda},\partial_{\xi_1}\partial_{u_3}^3g_{\lambda}$ and $\partial_{\xi_1}^2\partial_{u_3}^2g_{\lambda}$ (and permutations of order). Here we have
\begin{equation}\label{cthirdcomponent}
\begin{split}
&-\partial_{\xi_1}^2\partial_{u_3}^2g_{\lambda}(0,0)=-2(c_{13}-c_{23}),
\\
&-\partial_{\xi_1}\partial_{u_3}^3g_{\lambda}(0,0)=c_{13}k,
\\
&-\partial_{u_3}^4g_{\lambda}(0,0)=f.
\end{split}
\end{equation}
Combining the formulas (\ref{cfirstcomponent}) and (\ref{cthirdcomponent}) we deduce by integration by parts that
\begin{align*}
c\cdot a=&\frac{1}{3}\int_0^L 12(c_{13}-c_{23})\left[((w^{*}_1)^{\prime}w^{*}_3)^2-6c_{13}k (w^{*}_1)^{\prime}(w^{*}_3)^3-\lambda_0(w^{*}_3)^4\right]\,\dt
\\
=&\left(\frac{4(c_{13}-c_{23})(c_{13}k)^2}{c^2_{23}}+\frac{6(c_{13}k)^2}{c_{23}}-\frac{\lambda_0}{3}\right)\int_0^L\sin^4\left(\frac{2\pi}{L}t\right)\,\dt
\\
=&\left(\frac{4(c_{13}-c_{23})(c_{13}k)^2}{c^2_{23}}+\frac{6(c_{13}k)^2}{c_{23}}-\frac{\lambda_0}{3}\right)\frac{3L}{8}.
\end{align*}
This finishes the proof upon dividing by $a$ and plugging in the value of $\lambda_0$.
\end{proof}

\section{Proof of the main results}
Endowed with the lemmas of the previous section, we can now state and prove our main results. The first one concerns existence of a bifurcation branch from the identity at the critical force, and follows directly from the Crandall-Rabinowitz Theorem \ref{CR} and Lemma \ref{bifurcation}. In the statement we add some additional information, that we can deduce from Lemma \ref{abc}. First of all, the component of the bifurcation branch corresponding to the force varies, at least when we are close enough to $(I, f_{crit})$ in a left neighborhood of $f_{crit}$.  Since loss of stability for the identity happens when the force is decreasing, this corresponds to a {\it supercritical} bifurcation. Furthermore, the nontrivial stationary points have lower energy than the identity, for the corresponding value of the force along the bifurcation curve. 
Referring to the notation of Lemma \ref{abc}, these two results follow (since $b=0$) from the fact that $c<0$.

\begin{proposition}\label{definebranch}
Let $E_0^f$ be the functional defined in \eqref{1d-energy}, let $f_{crit}=\frac{(c_{13}k)^2}{c_{23}}-\frac{4\pi^2c_{12}}{L^2}$ and assume \eqref{ass_0}.
Then there exist $\delta, \eta >0$, an open neighborhood $U$ of $I$ in $W^{1,2}((0,L), SO(3))$ and a curve $(R, f):(-\delta, \delta)\to U \times (f_{crit}-\eta,f_{crit}]$ satisfying
\begin{enumerate}
 \item[(i)] $(R, f)(0)=(I, f_{crit})$;
 \item[(ii)] for all $s\neq 0$, $R(s)\neq I$ is a stationary point of the energy \eqref{1d-energy} with $f=f(s)$ and the boundary conditions \eqref{clamped}.
 \item[(iii)] for all $s\neq 0$, $E_0^{f(s)}(R(s))<E_0^{f(s)}(I)$.
\end{enumerate}
It holds furthermore
\begin{align}\label{firstorder}
\lim_{s\to 0}\frac{R(s)-\mathcal G(sw^{*})}{s}=0
\end{align}
in $C^{2}([0,L], SO(3))$, where $w^{*}$ is defined in \eqref{kernel}.
\end{proposition}

\begin{proof}
Thanks to Lemma \ref{bifurcation} we can apply Theorem \ref{CR} with $\lambda_0=f_{crit}$ to the functional $F$ in \eqref{defoperator}. For $\lambda(s)$ and $\psi(s)$ as in the statement of Theorem \ref{CR}, we then set 
\[
(R(s), f(s)):=(\mathcal G(sw^{*}+s\psi(s)), f_{crit}+\lambda(s))\,.
\]
With this, property (ii) follows from Remark \ref{equivequations}, while (i) is trivially satisfied.
By smoothness of $\mathcal G$ and since $\psi(0)=0$ we also have
\[
\lim_{s\to 0}\frac{\mathcal G(sw^{*}+s\psi(s))-\mathcal G(sw^{*})}{s}=0\,,
\]
that is \eqref{firstorder}. Since \cite[Remark 4.3 (iv)]{AmPr} gives $\lambda^{\prime}(0)=b$ and $\lambda^{\prime\prime}(0)=c$ with $b$ and $c$ as in \eqref{b}-\eqref{c}, from Lemma \ref{abc} we have that $f^{\prime}(0)=0$ and $f^{\prime \prime}(0)<0$, which implies
\begin{align}\label{parabola}
sf^{\prime}(s)<0
\end{align}
for $s \neq 0$ in small neighborhood of $0$. This proves that, up to possibly taking a smaller $\delta$ than the one given by Theorem \ref{CR}, $f(s)$ takes values in a left neighborhood $(f_{crit}-\eta,f_{crit}]$ of $f_{crit}$. Finally, in order to prove (iii) we introduce the function $e(s)=E_0^{f(s)}(R(s))-E_0^{f(s)}(I)$. Then $e(0)=0$ and by (\ref{parabola}), differentiation and the fact that $R(s)$ is a stationary point, we obtain
\begin{align*}
\frac{d}{ds}e(s)&=\underbrace{\langle DE_0^{f(s)}(R(s)),\partial_sR(s)\rangle}_{=0}-f^{\prime}(s)\int_0^L\langle (R(s)(t)-I)e_1,e_1\rangle\,\dt
\\
&=f^{\prime}(s)\int_0^L\langle (I-R(s)(t))e_1,e_1\rangle\,\dt\;
\begin{cases}
>0 &\mbox{if $s<0$,}
\\
<0 &\mbox{if $s>0$,}
\end{cases}
\end{align*}
where we denoted with $\partial_sR(s)$ the tangent vector to the curve $s\mapsto R(s)$ and used that the integrand is always nonnegative and not identically equal to zero.
This proves the last assertion. 
\end{proof}

The conditions $f^{\prime}(0)=0$ and $f^{\prime \prime}(0)<0$ imply that at least two nontrivial stationary points near to the straight configuration appear for $f< f_{crit}$. Furthermore, the set of stationary points can be shown to contain exactly three elements, in a suitably small $W^{1,2}$-neighborhood of $I$. In general, this neighborhood can be {\it smaller} than the set $U$ provided in Proposition \ref{definebranch}, since we cannot exclude folds in the bifurcation curve.
We give a precise statement of this property, after proving a preliminary lemma. It allows us to show that stationary points close to $I$ in $W^{1,2}$ are indeed also $C^2$-close and must therefore lie on the bifurcation curve provided by Theorem \ref{CR}.

\begin{lemma}\label{strongcontinuity}
Assume that $(R_n)_n\subset W^{1,2}((0,L),SO(3))$ is a sequence of solutions to (\ref{E-L-ODE}) and that $R_n\to R$ strongly in $W^{1,2}((0,L),SO(3))$. Then $R_n\to R$ also in the norm-topology of $C^2([0,L],SO(3))$.
\end{lemma}
\begin{proof}
By (\ref{E-L-ODE}) the sequence $A_n^{\prime}$ converges strongly in $L^1$ and hence $R_n\to R$ in $W^{2,1}((0,L),SO(3))$. Again by iteration and Sobolev embedding we find that $R_n\to R$ in $C^2([0,L],SO(3))$.	
\end{proof}

\begin{proposition}\label{isolated}
Let $E_0^f$ be defined as in (\ref{1d-energy}). There exist $\hat\eta>0$ and an open neighborhood $O$ of $I$ in $W^{1,2}((0,L), SO(3))$ such that for all $f\in (f_{crit}-\hat\eta, f_{crit})$ it holds
\begin{equation*}
\#\{R\in O\cap\mathcal{A}:\;R\text{ is a stationary point of }E_0^f\text{ in the sense of }\eqref{weakform}\}= 3.
\end{equation*}	
\end{proposition}
\begin{proof}
Let $\eta$ be as in Proposition \ref{definebranch}. Assume by contradiction that for each $n$ there exist $R^i_n\in\mathcal{A}$ ($i=1,\dots,4$) and $f_n$ such that $f_n \in (f_{crit}-\eta, f_{crit})$, $\|R^i_n-I\|_{W^{1,2}}\leq\frac{1}{n}$ and $R^i_n$ solves \eqref{weakform}. By Lemma \ref{regularity} we know that $R^i_n\in C^{\infty}([0,L],SO(3))$ and therefore it solves the strong form \eqref{E-L-ODE}. Therefore, Lemma \ref{strongcontinuity} implies that $R^i_n\to I$ also in the $C^2$-topology. Using Lemma \ref{inverse} we deduce that, for $n$ large enough, $(\mathcal G^{-1}(R^i_n), f_n)$ belongs to the neighborhood $U$ in \eqref{intorno}. This latter further implies that either $R^{i}_n=I$, or it must hold
$(R^i_n, f_n)=(R,f)(s^i_n)$ for $s^i_n \in (-\delta, \delta)$, where $(R,f)(s)$ is the curve and $(-\delta, \delta)$ is the interval provided by Proposition \ref{definebranch}. In particular, $f(s^i_n)$ takes the same value $f_n$ for all the $i$'s such that $R^i_n\neq I$, that is at least three times. It is not restrictive to assume that this is the case for $i=1$, $i=2$, and $i=3$.

We now fix $i \in \{1,2,3\}$, so that $R^i_n \neq I$. Since $R(s)=\mathcal G(sw^{*}+s\psi(s))$, $\psi(s)$ takes values in the topological complement $Z$ of ${\rm span}\{w^{*}\}$, and we have by construction $R^i_n\to I$ in $W^{1,2}((0,L), SO(3))$, it must hold $s^i_n\to 0$ when $n\to +\infty$. 
Since $f(s^1_n)=f(s^2_n)=f(s^3_n)$, applying for instance Rolle's theorem twice, we can then construct a sequence $\hat{s}_n \to 0$ with $f^{\prime\prime}(\hat{s}_n)\to 0$. Since $f$ is smooth by Lemma \ref{diff} and Theorem \ref{CR}, this contradicts the fact that $f^{\prime\prime}(0)<0$ proved in Proposition \ref{definebranch}. With this, for $n_0$ large enough, setting
\begin{align*}
O:=\left\{R \in W^{1,2}((0,L), SO(3)): \|R-I\|_{W^{1,2}} < \frac{1}{n_0}\right\}\, 
\end{align*}
we get
\begin{align}\label{inequality}
\#\{R\in O\cap\mathcal{A}:\;R\text{ is a stationary point of }E_0^f\text{ in the sense of }\eqref{weakform}\}\le 3\,.
\end{align}
Since $f$ is smooth, monotone decreasing for $s>0$ and monotone increasing for $s<0$ by \eqref{parabola}, there exists $0<\hat\eta\le \eta$ such that $f(s)$ takes each value in $(f_{crit}-\hat\eta, f_{crit})$ twice in a neighborhood of $0$. Then, the converse inequality to \eqref{inequality} follows from Proposition \ref{definebranch} (ii). 
\end{proof}

Relying on Theorem \ref{CRII} and Lemma \ref{abc}, we now prove that the curve $R(s)$ provided by Proposition \ref{definebranch} consists, at least for $|s|$ small enough, of {\it  isolated local minimizers} of the energy \eqref{1d-energy} for $f=f(s)$.
\begin{theorem}\label{localopt}
Let $E_0^f$ be the functional defined in \eqref{1d-energy} and let $f_{crit}=\frac{(c_{13}k)^2}{c_{23}}-\frac{4\pi^2c_{12}}{L^2}>0$ and assume \eqref{ass_0}.
Then there exist $\delta, \eta >0$, an open neighborhood $U$ of $I$ in $W^{1,2}((0,L), SO(3))$ and a curve $(R, f):(-\delta, \delta)\to U \times (f_{crit}-\eta,f_{crit}]$ satisfying
\begin{enumerate}
 \item[(i)] $(R, f)(0)=(I, f_{crit})$;
 \item[(ii)] for all $s\neq 0$, $R(s)\neq I$ is an isolated local minimizer of the energy \eqref{1d-energy} with $f=f(s)$ and the boundary conditions \eqref{clamped}.
 \item[(iii)] for all $s\neq 0$, $E_0^{f(s)}(R(s))<E_0^{f(s)}(I)$.
\end{enumerate}
Furthermore, $R(s)$ satisfies \eqref{firstorder}.
\end{theorem}
\begin{proof}
Considering the curve $(R(s), f(s))$ provided by Proposition \ref{definebranch}, we only have to prove that, up to possibly reducing the interval $(-\delta, \delta)$, property (ii) is satisfied. Setting $\varphi(s):=\mathcal G^{-1}(R(s))$ and considering the auxiliary energy $E^{f(s)}_0(\varphi(s))$ defined in \eqref{angleenergy}, taking into account \eqref{frechet} and Remark \ref{equivequations} it suffices to prove that, for $|s|$ small enough, there exists $c(s)>0$ such that
\begin{equation*}
D^2E^{f(s)}_0(\varphi(s))[w,w]\geq c(s)\|w\|^2_{W^{1,2}}
\end{equation*}
for all $w\in W^{1,2}_0((0,L),\R^3)$. 
In order to prove the strict positivity of the second differential of the energy, we argue by contradiction. We fix $s$ and assume that for each $n\in\mathbb{N}$ there exists $w_n\in W^{1,2}_0((0,L),\R^3)$ such that $\|w_n\|_{W^{1,2}}=1$ and $D^2E_0^{f(s)}(\varphi(s))[w_n,w_n]\leq \frac{1}{n}$. Passing to a subsequence we have that $w_n\rightharpoonup w$ in $W_0^{1,2}((0,L),\R^3)$. Observe that
\begin{align}\label{secondder}
D^2E_0^{f(s)}(\varphi(s))[w_n,w_n]=&\int_0^L\partial^2_{\xi}g_{f(s)}(\varphi(s),\partial_t \varphi(s))[w_n^{\prime},w_n^{\prime}]\,\dt\nonumber
\\
&+\int_0^L\partial^2_ug_{f(s)}(\varphi(s),\partial_t \varphi(s))[w_n,w_n]+2\partial_{\xi}\partial_{u}g_{f(s)}(\varphi(s),\partial_t \varphi(s))[w_n,w_n^{\prime}]\,\dt\,,
\end{align}
where for the sake of notation we denote with $\partial_t \varphi(s)$ the derivative of the absolutely continuous function $\varphi(s)$ with respect to $t\in [0,L]$. The last integral in \eqref{secondder} is continuous with respect to weak convergence of $w_n$ in $W^{1,2}((0,L),\R^3)$ by uniform convergence of $w_n$ and weak-convergence properties of products. Considering the first integral, it holds that
\begin{equation}\label{elliptic}
\partial^2_{\xi}g_{f(s)}(\varphi(s),\partial_t \varphi(s))\to \mathbf{C}
\end{equation}
uniformly on $[0,L]$ when $s\to 0$, so that the associated quadratic form becomes globally positive definite for $|s|$ small enough. We conclude that $w\mapsto D^2E_0^{f(s)}(\varphi(s))[w,w]$ is weakly lower semicontinuous on $W^{1,2}_0((0,L),\R^3)$, provided $|s|$ is small enough. Consequently 
\begin{equation}\label{nonpositive}
D^2E_0^{f(s)}(\varphi(s))[w,w]\leq 0.
\end{equation}
We argue that $w\neq 0$. Indeed, when $w=0$ then $w_n\to 0$ uniformly on $[0,L]$ and $w_n^{\prime}\rightharpoonup 0$ weakly in $L^2((0,L))$. Passing to the limit in the formula \eqref{secondder} and using the weak continuity of the second integral, from our assumptions on $w_n$ we infer that
\begin{equation*}
0\geq \limsup_n\int_0^L\partial^2_{\xi}g_{f(s)}(\varphi(s),\partial_t \varphi(s))[w_n^{\prime},w_n^{\prime}]\,\dt.
\end{equation*}
Since $\partial^2_{\xi}g_{f(s)}(\varphi(s),\partial_t \varphi(s))$ is globally positive definite, the above inequality implies that $w_n\to 0$ strongly in $W^{1,2}((0,L),\R^3)$ which contradicts the normalization of $\|w_n\|_{W^{1,2}}$. Thus $w\neq 0$.

Now consider the auxiliary problem
\begin{equation*}
\inf\{D^2E_0^{f(s)}(\varphi(s))[w,w]:\;w\in W^{1,2}_0((0,L),\R^3),\,\|w\|^2_{L^2}=1\}.
\end{equation*}   
Again for small $|s|$ weak lower semicontinuity implies that the above problem has a solution. Combining (\ref{nonpositive}) with rescaling, we know that any minimizer $\bar{w}(s)$ satisfies
\begin{equation}\label{sign}
D^2E_0^{f(s)}(\varphi(s))[\bar{w}(s),\bar{w}(s)]\leq 0.
\end{equation}
By constrained minimality and symmetry of the second variation there exists a Lagrange-multiplier $\bar{\mu}(s)\in\R$ such that, for all $g\in W^{1,2}_0((0,L),\R^3)$, 
\begin{equation}\label{lagrange}
D^2E_0^{f(s)}(\varphi(s))[\bar{w}(s),g]+\bar{\mu}(s)\int_0^L\langle\bar{w}(s),g\rangle\,\dt=0.
\end{equation}
Choosing $g=\bar{w}(s)$ in the above inequality we deduce from (\ref{sign}) that 
\begin{equation}\label{signmu}
\bar{\mu}(s)=-D^2E_0^{f(s)}(\varphi(s))[\bar{w}(s),\bar{w}(s)]\geq 0.
\end{equation}
Note that by \eqref{minussign} the equation \eqref{lagrange} is the weak formulation of the following system:
\begin{equation}\label{stronglagrange}
D_{\varphi}F(\varphi(s),f(s))\bar{w}(s)=\bar{\mu}(s) \bar{w}(s),
\end{equation} 
where the operator $D_{\varphi}F(\varphi(s),f(s))$ is given explicitly in \eqref{derivativeoperator}. By \eqref{elliptic}, for $|s|$ small enough, the matrix $\partial^2_{\xi}g_{f(s)}(\varphi(s)(t),\partial_t\varphi(s)(t))$ is invertible for all $t\in [0,L]$. Furthermore, the inverse matrix depends smoothly on $t$, since, for all $s$, the function $t\mapsto\varphi(s)(t)\in C^{\infty}([0,L],\R^3)$ by Remark \ref{equivequations}. Then by a similar bootstrap argument as for Lemma \ref{regularity} one proves that $\bar{w}(s)\in C^{\infty}([0,L],\R^3)$. Hence we can replace \eqref{lagrange} by its strong formulation \eqref{stronglagrange}. 

Now we want to apply Theorem \ref{CRII} with $\lambda_{0}=f_{crit}$.  To this end we set $K:C^2_0([0,L],\R^3)\to C([0,L],\R^3)$ as the natural embedding $K\varphi=\varphi$. Note that the assumption that $0$ is a $K$-simple eigenvalue has been verified in the proof of Lemma \ref{bifurcation}. Having in mind Lemma \ref{quasiuniqueness}, we first prove that $\bar{\mu}(s)\to 0$ when $s\to 0$. Rewriting \eqref{signmu} via
\begin{align*}
\bar{\mu}(s)=(D^2E_0^{f(0)}(0)-D^2E_0^{f(s)})(\varphi(s))[\bar{w}(s),\bar{w}(s)]-D^2E_0^{f(0)}(0)[\bar{w}(s),\bar{w}(s)],
\end{align*}
we observe that the first difference vanishes by \eqref{frechet} and the convergence properties of $s\mapsto(\varphi(s),f(s))$. Since $D^2E_0^{f(0)}(0)$ is positive semidefinite and $\bar{\mu}(s)\geq 0$, we deduce from the above equality that $\bar{\mu}(s)\to 0$ when $s\to 0$. \\
The convergence of $\mu(s)$ to $0$ eventually implies
\[
\lim_{s\to 0}D^2E_0^{f(0)}(0)[\bar{w}(s),\bar{w}(s)]=0
\]
The general expression for $D^2E_0^{f(0)}(0)[w,w]$ can be easily computed from, e.g.\ , \eqref{secondder}, inserting $\varphi(0)=0$. We have
\begin{align*}
D^2E_0^{f(0)}(0)[w,w]=\int_0^L\langle \mathbf{C}w^{\prime},w^{\prime}\rangle-c_{13}k(w_1w_3^{\prime}-w_1^{\prime}w_3)+f_{crit}(w_2^2+w_3^2)\,\mathrm{d}t\,.
\end{align*}
Since $\mathbf{C}$ is positive definite and $\frac{w^{*}}{\|w^{*}\|_2}$ is the unique normalized vector in the nullspace of $D^2E_0^{f(0)}(0)$, a standard argument implies then that $\bar{w}(s)$ converges strongly in $W^{1,2}_0$ to $\frac{w^{*}}{\|w^{*}\|_2}$. Writing
\[
\bar{w}(s)=c_1(s)w^*+c_2(s)z(s)
\]
with $c_1(s)$, $c_2(s)\in \R$ and $z(s)$ in the topological complement $Z$ of $w^*$, it follows that $c_1(s)$ stays bounded away from $0$, so that (up to possibly dividing by $c_1(s)$), we can assume that $\bar{w}(s)-w^* \in Z$. With this, Lemma \ref{quasiuniqueness} implies that $\bar \mu(s)$ and $\bar w(s)$ coincide with the curves $\mu(s)$ and $w(s)$ solving \eqref{eigenprob}, for $K$ being the natural embedding of $C^2_0([0,L],\R^3)$ into $ C([0,L],\R^3)$.

Moreover, for every possible curve of eigenvalues $\gamma(\lambda)$ of $D_{\varphi}F(0,\lambda)$ we have $\gamma^{\prime}(\lambda_0)\leq 0$ as the second variation of the energy $E_0^f$ is positive definite for $f>\lambda_0$ (see Theorem \ref{idstable?}). Here we remark that the operator $-F$ corresponds to the first differential of the energy $E_0^f$ in the sense of (\ref{minussign}), so that one has to take care of a sign change. By Theorem \ref{CRII} we conclude that $\gamma^{\prime}(\lambda_0)<0$. By \eqref{parabola} we know that $-s\lambda^{\prime}(s)=-sf^{\prime}(s)>0$ for $|s|$ small enough, and consequently Theorem \ref{CRII} and Remark \ref{quasiuniqueness} imply $\bar{\mu}(s)<0$. This contradicts (\ref{signmu}).
\end{proof}

\begin{remark}\label{shape}
The local minimizers $R(s)$ are hemihelical minimizers in the following sense. According to \eqref{firstorder} they are, for $|s|<<1$ (that is for $f\approx f_{crit}$), close in the $C^2$ topology to the function $t\mapsto \mathcal{G}(sw^{*})(t)$. This one shows clearly an hemihelical shape with the occurrence of a perversion. In Figure \ref{fig1} we plot the shape of this asymptotic expansion. 
\end{remark}

\begin{figure}[h]
\includegraphics[trim=5cm 9cm 2cm 9cm, scale=0.5]{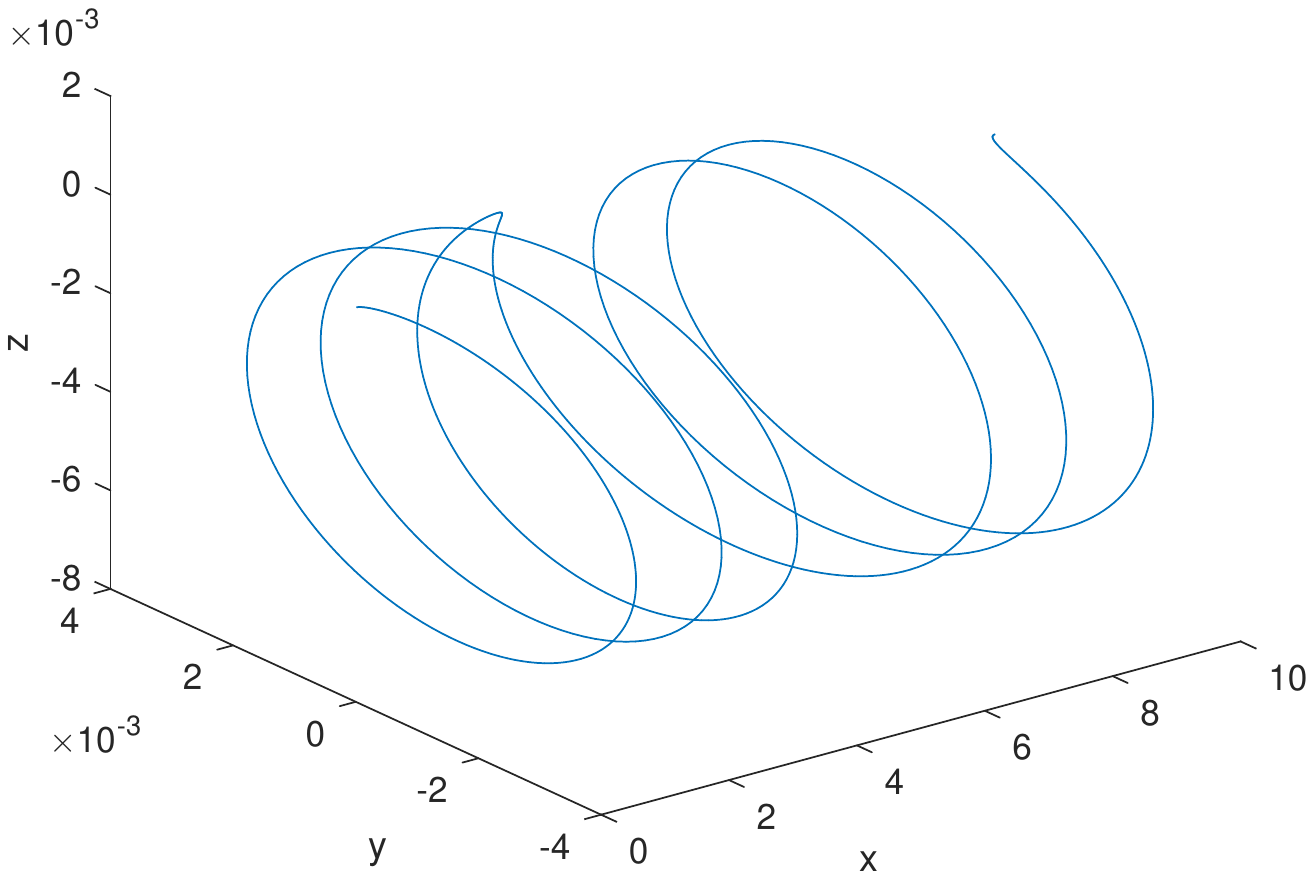}%
\caption{Shape of the centerline of the function $\mathcal{G}(sw^*)$. The parameters are given by $c_{12}=4.0848$, $c_{13}=0.0065$ and $c_{23}=0.0087$. The intrinsic curvature is $k=375$ and the critical force is $f_{crit}=687$. We chose $s=0.02$.}
	\label{fig1}
\end{figure}

Our final result states that hemihelical local minimizers appear indeed as limits of local minimizers of the $3$-dimensional energy \eqref{defenergy} as the size of the cross-section vanishes. 
\begin{theorem}\label{hemihelix}
Assume that the energy 	$E_{h}$ defined in (\ref{defenergy}) is lower-semicontinuous with respect to weak convergence in $W^{1,2}(\Omega,\R^3)$, and consider $E_0^f$ as in \eqref{1d-energy}. Set further $f_{crit}=\frac{(c_{13}k)^2}{c_{23}}-\frac{4\pi^2c_{12}}{L^2}$ and assume \eqref{ass_0}. Fix an arbitrary open neighborhood $U$ of $I$ in $W^{1,2}((0,L), SO(3))$. 
\\
Then there exists $\eta>0$ such that, for all $f\in (f_{crit}-\eta, f_{crit})$, we may find two nontrivial local minimizers $R_f^1, R_f^2 \in U$ of $E_0^f$ and two sequences $v^1_{h}$, $v^2_{h}$ of local minimizers  of $E_{h}$ in $\mathcal{A}_{h}$, which satisfy 
\[
v^\ell_{h}\to \int_0^{x_1}R^\ell_f(t)e_1\,\dt\,,  \quad \nabla_{h}v^\ell_{h}\to R^\ell_f\qquad (\ell=1,2)
\]
strongly in $W^{1,2}(\Omega,\R^3)$, and strongly in $L^2(\Omega,\mathbb{M}^{3\times 3})$, respectively. 
\end{theorem}

\begin{proof}
The result immediately follows combining Theorem \ref{conv-loc} and Theorem \ref{localopt}, upon noticing that, by \eqref{parabola}, there exists $\eta>0$ such that the function $f(s)$ provided by Theorem \ref{localopt} takes each value in $(f_{crit}-\eta, f_{crit})$ twice in a neighborhood of $0$.
\end{proof}

\noindent {\bf Acknowledgements} The authors would like to thank A. Mielke for some useful discussions on bifurcation phenomena. The work of MC was partially supported by the DFG Collaborative Research Center TRR 109, ``Discretization in Geometry and Dynamics''.

\end{document}